\documentclass[12pt]{amsart}
\usepackage{amsmath}
\usepackage{amscd}
\usepackage{amssymb}
\usepackage{amsfonts}

\newtheorem{theorem}{Theorem}[section]
\newtheorem{lemma}[theorem]{Lemma}

\theoremstyle{definition}

\theoremstyle{remark}
\newtheorem{remark}[theorem]{Remark}
\numberwithin{equation}{section}

\begin{document}

\title[Sharp upper diameter bounds for compact shrinkers]
{Sharp upper diameter bounds for compact shrinking Ricci solitons}
\author{Jia-Yong Wu}
\address{Department of Mathematics, Shanghai University, Shanghai 200444, China}
\email{wujiayong@shu.edu.cn}
\thanks{}
\dedicatory{}
\date{\today}


\begin{abstract}
We give a sharp upper diameter bound for a compact shrinking Ricci soliton in terms
of its scalar curvature integral and the Perelman's entropy functional. The sharp
cases could occur at round spheres. The proof mainly relies on a sharp logarithmic
Sobolev inequality of gradient shrinking Ricci solitons and a Vitali-type covering
argument.
\end{abstract}
\maketitle

\section{Introduction}\label{Int1}
A complete Riemannian metric $g$ on a smooth $n$-dimensional manifold $M$ is called a
\emph{Ricci soliton} if there exists a smooth vector field $V$ on $M$ such that
the Ricci curvature $\text{Ric}$ of the metric $g$ satisfies
\[
\mathrm{Ric}+\tfrac 12\mathcal{L}_Vg=\lambda g
\]
for some real constant $\lambda$, where $\mathcal{L}_V$ denotes the Lie derivative in
the direction of $V$. A Ricci soliton is called \emph{shrinking},
\emph{steady} or \emph{expanding}, if $\lambda>0$, $\lambda=0$ or $\lambda<0$,
respectively. When $V=\nabla f$ for some smooth function $f$ on $M$, then the
Ricci soliton becomes the \emph{gradient Ricci soliton}
\[
\mathrm{Ric}+\mathrm{Hess}\,f=\lambda g,
\]
where $\text{Hess}\,f$ denotes the Hessian of $f$. The function $f$ is often called a
\emph{potential function}. Perelman \cite{[Pe]} proved that every compact Ricci soliton
is necessarily gradient. For $\lambda>0$, scaling metric $g$, we can normalize
$\lambda=\frac 12$ so that
\begin{align}\label{Eq1}
\mathrm{Ric}+\mathrm{Hess}\,f=\frac 12g.
\end{align}
We set $\mathrm{Ric}_f:=\mathrm{Ric}+\mathrm{Hess}\,f$, which is customarily called
the Bakry-\'Emery Ricci tensor \cite{[BE]}. $\mathrm{Ric}_f$ is an important
geometric quantity, which can be used to show that the Ricci flow is a gradient
flow of the Perelman's $\mathcal{F}$-functional \cite{[Pe]}. In the whole paper, we
let a triple $(M, g, f)$ denote an $n$-dimensional complete gradient shrinking Ricci
soliton. As in \cite{[LLW]}, normalizing $f$ by adding a constant in \eqref{Eq1},
without loss of generality, we may assume \eqref{Eq1} simultaneously satisfies
\begin{equation}\label{Eq2}
\mathrm{R}+|\nabla f|^2=f \quad\mathrm{and}\quad (4\pi)^{-\frac n2}\int_Me^{-f} dv=e^{\mu},
\end{equation}
where $\mathrm{R}$ is the scalar curvature of $(M,g)$ and $\mu=\mu(g,1)$ is the entropy
functional of Perelman \cite{[Pe]}; see the explanation in Section \ref{sec2}.
Note that $\mu$ is a finite constant for a fixed complete gradient shrinking Ricci
soliton. From Lemma 2.5 in \cite{[LLW]}, we know that $e^{\mu}$ is almost equivalent
to the volume of the geodesic ball $B(p_0,1)$ with radius $1$ and center $p_0$.
That is,
\begin{equation}\label{equiv}
\frac{(4\pi)^{\frac n2}}{e^{2^{4n+7}}}\leq\frac{V(p_0,1)}{e^{\mu}}\leq (4\pi)^{\frac n2}e^n,
\end{equation}
where $V(p_0,1)$ denotes the volume of $B(p_0,1)$. Here $p_0\in M $ is a point where
$f$ attains its infimum, which always exists on the complete (compact or not-compact) gradient
shrinking Ricci soliton $(M,g,f)$; see \cite{[HaMu]}. By the Chen's argument \cite{[Chen]},
we know that $\mathrm{R}\ge0$. By the Pigola-Rimoldi-Setti work \cite{[PiRS]}, we further
know that $\mathrm{R}>0$ unless $(M,g,f)$ is the Euclidean Gaussian shrinking
Ricci soliton $(\mathbb{R}^n, g_E, \frac{|x|^2}{4})$.

\vspace{.1in}

Gradient shrinking Ricci solitons can be regarded as a natural extension of Einstein
manifolds. They play an important role in the Ricci flow as they correspond to some
self-similar solutions and often rise as singularity models of the Ricci flow \cite{[Ham]}.
They are also viewed as critical points of the Perelman's entropy functional \cite{[Pe]}.
At present, one of most important project is the classification of complete gradient
shrinking Ricci solitons. For dimension $2$, the classification is complete \cite{[Ha88]}.
In particular, Hamilton proved that every $2$-dimensional compact shrinking Ricci
soliton must be Einstein. For dimension $3$, Ivey \cite{[Ivey]} proved that any compact
shrinking Ricci solitons are still Einstein; the non-compact case is a little complicated
and has been completely classified by the work of \cite{[Pe]}, \cite {[NW]} and \cite{[CCZ]}.
However for the higher dimensions, even $n=4$, the classification remains open, though much
progress has been made; see, e.g., \cite{[CCq]}, \cite{[Cati]}, \cite{[ChWa]}, \cite{[FG1]},
\cite{[MS]}, \cite{[MuWa]}, \cite{[MuWa2]}, \cite{[NW]}, \cite{[PW]}, \cite{[WWW]} and \cite{[ZZH]}.

\vspace{.1in}

On the other hand, Ivey \cite{[Ivey]} confirmed that any compact gradient steady or expanding
Ricci solitons are Einstein. Therefore the shrinking cases are the only possible non-Einstein
compact Ricci solitons. In fact for $n=4$, Cao \cite{[Ca96]}, Koiso \cite{[Ko]}, Wang and Zhu
\cite{[WaZ]} successfully constructed non-Einstein examples of compact K\"ahler shrinking Ricci
solitons. At present, all of known compact shrinking Ricci solitons are K\"ahler. It remains
an interesting question whether there exists a non-K\"ahler compact shrinking Ricci soliton.
Derdzi\'nski \cite{[De]} proved that every compact shrinking Ricci soliton has a finite
fundamental group (the non-compact case due to Wylie \cite{[Wy]}). We refer to further
related work in \cite{[BW]}, \cite{[CaTr]}, \cite{[Cat]}, \cite{[Cao]} and references therein.

\vspace{.1in}

In this paper, we will study the diameter estimate for a compact (without boundary) gradient
shrinking Ricci soliton. We will give a sharp upper diameter bound in terms of the
$L^{\frac{n-1}{2}}$-norm of the scalar curvature and the Perelman's entropy functional.
On a compact shrinking Ricci soliton $(M,g,f)$, the diameter of $M$ is defined by
\[
\mathrm{diam}(M):=\max\left\{dist(p,q)|\,\,\,\forall\,p,\, q\in M\right\},
\]
where $dist(p,q)$ denotes the geodesic distance between points $p$ and $q$. Recently,
there has been lots of effort to estimate the diameter of gradient shrinking Ricci
solitons. In \cite{[FS]}, Futaki and Sano got a lower diameter bound for non-Einstein
compact shrinking Ricci solitons, which was then sharpened by Andrews and Ni \cite{[AN]},
and Futaki, Li and Li \cite{[FLL]}. In \cite{[FG]}, Fern\'andez-L\'opez and Garc\'ia-R\'io
studied some properties of geodesics on Ricci solitons and obtained many lower diameter
bounds for compact gradient solitons in terms of extremal values of the potential function,
the scalar curvature and the Ricci curvature on unit tangent vectors. Motivated by the
classical Myers' theorem, Fern\'andez-L\'opez and Garc\'ia-R\'io \cite{[FG08]} proved a
Myers' type theorem on Riemannian manifolds when $\mathrm{Ric}_f$ is bounded below by a
positive constant and $|\nabla f|$ is bounded. Later, Limoncu \cite{[Lim]} and Tadano
\cite{[Tad2]} respectively gave an explicit upper diameter bound for such manifolds,
which was sharpened by the author \cite{[Wu]}. When the bound of $|\nabla f|$ is
replaced by the bound of $f$, many upper diameter bounds were studied by Wei and Wylie
\cite{[WW]}, Limoncu \cite{[Lim2]} and Tadano \cite{[Tad1]}, etc. For more related
results, the interested reader can refer to \cite{[BQ]}, \cite{[FG2]}, \cite{[Lot]}
and the references therein.

\vspace{.1in}

In another direction, Bakry and Ledoux \cite{[BL]} applied a sharp Sobolev inequality
of manifolds to give an alternative proof of the Myers' diameter estimate, which indicates
that some functional inequalities of manifolds may suffice to produce an upper diameter
bound of manifolds. Similar idea also appeared in the other literatures. For example,
Topping \cite{[To3]} applied the Michael-Simon Sobolev inequality to obtain an upper
diameter bound for a closed connected manifold immersed in $\mathbb{R}^n$ in terms of
its mean curvature integral. Topping's result was later generalized by Zheng and the
author \cite{[WZ]} to a general ambient space.

\vspace{.1in}

The above method is also suitable to the Ricci flow setting. In \cite{[To2]}, Topping
applied the Perelman's $\mathcal{W}$-functional to get an upper diameter bound for a
compact manifold evolving under the Ricci flow. Here the upper bound depends on the
scalar curvature integral under the evolving metric and some geometric quantities
with the initial metric. Inspired by Topping's argument, Zhang \cite{[Zhq]} applied
the uniform Sobolev inequality along the Ricci flow to obtain an upper diameter
bound in terms of the scalar curvature integral, volume and Sobolev constants
(or positive Yamabe constants) under the Ricci flow. Meanwhile, he proved a sharp
lower bound for the diameters, which depends on the initial metric, time and the
scalar curvature integral. We would like to mention that Zhang's argument is also
suitable to stationary manifolds.

\vspace{.1in}

Inspired by the work of Topping \cite{[To2]} and Zhang  \cite{[Zhq]}, in this paper
we are able to prove a sharp upper diameter bound for a compact shrinking Ricci soliton
without any assumption. Our result gives an explicit coefficient of the diameter estimate
in terms of the scalar curvature integral and the Perelman's entropy functional.

\begin{theorem}\label{Main}
Let $(M,g, f)$ be an $n$-dimensional $(n\geq3)$ compact gradient shrinking Ricci soliton satisfying
\eqref{Eq1} and \eqref{Eq2}. Then there exists a constant $c_1(n,\mu)$ depending  on
$n$ and $\mu$ such that
\[
\mathrm{diam}(M)\le c_1(n,\mu)\int_M\mathrm{R}^{\frac{n-1}{2}}dv,
\]
where $\mathrm{R}$ is the scalar curvature of $(M,g,f)$ and $\mu=\mu(g,1)$
is the Perelman's entropy functional. In particular, we can take
\[
c_1(n,\mu)=4\max\left\{w^{-1}_n, (4\pi)^{-\frac n2}e^{2^n\cdot17-\mu-n}\right\},
\]
where $w_n$ is the volume of the unit $n$-dimensional ball in $\mathbb{R}^n$.
\end{theorem}
\begin{remark}
The theorem is also suitable to positive Einstein manifolds. The exponent $\frac{n-1}{2}$
of the scalar curvature is sharp. Indeed, we consider the round $n$-sphere
$\mathbb{S}^n(r)$ of radius $r$ with the canonical metrics $g_0$ and let
$f=\mathrm{constant}$. Then $(\mathbb{S}^n(r), g_0, f)$ is a trivial compact
gradient shrinking Ricci soliton. Its diameter is almost equivalent to $r$,
i.e., $\mathrm{diam}_{g_0}(M)\approx r$; while the scalar curvature
$\mathrm{R}(g_0)\approx r^{-2}$. If we scale metric $g_0$ to be $g\approx r^{-2}g_0$
such that $\mathrm{Ric}(g)=\frac 12 g$, then by \eqref{equiv},
\[
e^{\mu(g,1)}\approx V_g(p,1)=\frac{V_{g_0}(p,r)}{r^n}=c(n),
\]
where $V_g(p,1)$ denotes the volume of ball $B(p,1)$ with respect to metric $g$. This
indicates that coefficient $c_1(n,\mu)$ only depends on $n$ and the right hand side of the
diameter estimate in the theorem can be easily computed to be $c(n)r$.
\end{remark}

\begin{remark}
We omit the discussion about the optimal choice of $c_1(n,\mu)$. One might get a sharper
constant $c_1(n,\mu)$ by choosing a better cut-off function in Section \ref{sec3}.
\end{remark}

We would like to point out that previous diameter estimates for gradient shrinking Ricci
solitons mainly relies on pointwise conditions of geometric quantities; see, e.g.,
\cite{[Lim],[Lim2]}, \cite{[Tad1], [Tad2]}, \cite{[WW]} and \cite{[Wu]}. Our estimate is
valid in the integral sense and it seems to be weaker than before. In \cite{[MuWa]},
Munteanu and Wang proved an upper diameter bound for a compact shrinking Ricci soliton
in terms of its injectivity radius. Our estimate depends on the scalar curvature integral
and the Perelman's entropy functional, and it may be a more feasible dependence on
geometric quantities.

\vspace{.1in}

The trick of proving Theorem \ref{Main} stems from \cite{[To2]}, but we need to carefully
examine the explicit coefficient of the diameter bound in terms of the scalar curvature
integral. Our argument is divided into three steps. First, we apply a sharp logarithmic
Sobolev inequality and a proper cut-off function to get a new functional inequality, which
is related to the maximal function of scalar curvature and the volume ratio (see Theorem
\ref{logeq}). We mention that the sharp logarithmic Sobolev inequality is a key inequality
in our paper, which was proved by Li, Li and Wang \cite{[LLW]} for compact Ricci solitons
and then extended by Li and Wang \cite{[LiWa]} to the non-compact case. Second, we use the
functional inequality to give an alternative theorem, which states that the maximal function
of scalar curvature and the volume ratio cannot be simultaneously smaller than a fixed
constant on a geodesic ball of shrinking Ricci soliton (see Theorem \ref{maxirati}).
Third, we apply the alternative theorem and a Vitali-type covering lemma to give the
diameter estimate.

\vspace{.1in}

The structure of this paper is as follows. In Section \ref{sec2}, we recall some basic results
about gradient shrinking Ricci solitons. In particular, we rewrite the Li-Wang's logarithmic
Sobolev inequality \cite{[LiWa]} as a functional inequality by choosing a proper cut-off function.
In Section \ref{sec3}, we use the functional inequality to give an alternative theorem. In Section
\ref{sec4}, we apply the alternative theorem to prove Theorem \ref{Main}.

\vspace{.1in}

\textbf{Acknowledgement}.
The author thanks Peng Wu for helpful discussions. The author also thanks the referee for valuable
comments and suggestions, which helped to improve the paper. This work is supported by the
NSFC (11671141) and the Natural Science Foundation of Shanghai (17ZR1412800).


\section{Background}\label{sec2}
In this section, we recall some basic results about gradient shrinking Ricci solitons
and give an explanation why \eqref{Eq2} can be suitable to \eqref{Eq1}. We also rewrite
the Li-Wang's logarithmic Sobolev inequality \cite{[LiWa]} to a new functional inequality
relating the maximal function of scalar curvature and the volume ratio. For more properties
about Ricci solitons, the interested reader refer to the survey \cite{[Cao]}.

\vspace{.1in}

In this paper, we concentrate on compact shrinking Ricci solitons, however the following
results are also suitable to the non-compact case. By Hamilton \cite{[Ham]},
\eqref{Eq1} gives that
\[
\mathrm{R}+\Delta f=\frac n2,\quad 2\mathrm{Ric}(\nabla f)=\nabla\mathrm{R}
\]
and
\[
\nabla(\mathrm{R}+|\nabla f|^2-f)=0.
\]
Adding $f$ by a constant if necessary, we have that
\begin{equation}\label{condition}
\mathrm{R}+|\nabla f|^2=f.
\end{equation}
Combining the above equalities gives
\begin{equation}\label{identitycond}
2\Delta f-|\nabla f|^2+\mathrm{R}+f-n=0.
\end{equation}

By Cao-Zhou \cite{[CaZh]} and Haslhofer-M\"uller \cite{[HaMu]}, we have a precise asymptotic
estimate of $f$.
\begin{lemma}\label{potenesti}
Let $(M,g, f)$ be an $n$-dimensional complete non-compact gradient shrinking Ricci soliton
satisfying \eqref{Eq1} and \eqref{condition}. Then there exists a point $p_0\in M $ where
$f$ attains its infimum (may be not unique). Moreover, $f$ satisfies
\[
\frac 14\left[\left(r(x,p_0)-5n\right)_{+}\right]^2\le f(x)\le\frac 14\left(r(x,p_0)+\sqrt{2n}\right)^2,
\]
where $r(x,p_0)$ is a distance function from $p_0$
to $x$, and $a_+=\max\{a,0\}$ for $a\in \mathbb{R}$.
\end{lemma}
\begin{remark}\label{potenest}
In view of the flat Euclidean space $(\mathbb{R}^n,\delta_{ij})$ with
$f=|x|^2/4$, the above leading term $\frac 14 r^2(x,p_0)$ is optimal.
\end{remark}

For an $n$-dimensional complete Riemannian manifold $(M,g)$, the definition of the
Perelman's $\mathcal{W}$-entropy functional \cite{[Pe]} is
\[
\mathcal{W}(g,\varphi,\tau)
:=\int_M\left[\tau\Big(|\nabla \varphi|^2+\mathrm{R}\Big)+\varphi-n\right](4\pi\tau)^{-n/2}e^{-\varphi}dv
\]
for some $\varphi\in C^\infty(M)$ and $\tau>0$, provided this functional
is finite. The Perelman's $\mu$-entropy functional \cite{[Pe]} is defined by
\[
\mu(g,\tau):=\inf\left\{\mathcal{W}(g,\varphi,\tau)\Big|\varphi\in C^\infty(M)\,\,\,\mathrm{with}\,\,\, \int_M(4\pi\tau)^{-n/2}e^{-\varphi}dv=1\right\}.
\]
In general, the minimizer of $\mu(g,\tau)$ may not exist on non-compact manifolds.
However, by Lemma \ref{potenesti}, the above definitions are both well defined on
non-compact gradient shrinking Ricci solitons and many integrations by parts still
hold; see the explanation in \cite{[HaMu]}. Moreover, Carrillo and Ni \cite{[CaNi]}
proved that potential function $f$ is always a minimizer of $\mu(g,1)$, up to adding
$f$ by a constant. That is, for a constant $c$ with
\[
\int_M(4\pi)^{-n/2}e^{-(f+c)}dv=1,
\]
we have
\begin{align*}
\mu(g,1)&=\mathcal{W}(g,f+c,1)\\
&=\int_M\Big(|\nabla f|^2+\mathrm{R}+(f+c)-n\Big)(4\pi)^{-n/2}e^{-(f+c)}dv\\
&=\int_M\Big(2\Delta f-|\nabla f|^2+\mathrm{R}+(f+c)-n\Big)(4\pi)^{-n/2}e^{-(f+c)}dv\\
&=c.
\end{align*}
Here we used the integration by parts in the above third line because $f$ is
uniformly equivalent to the distance function squared and it guarantees
the integration by parts on non-compact manifolds; see \cite{[HaMu]}. We also
used \eqref{identitycond} in the above last line. Therefore we can assume that
\eqref{Eq1} satisfies \eqref{Eq2} in the introduction.

\vspace{.1in}

Carrillo and Ni \cite{[CaNi]} proved that $\mu(g,1)$ is the optimal logarithmic Sobolev
constant on complete shrinking Ricci soliton $(M,g,f)$ for scale one. Later, Li, Li and Wang
\cite{[LLW]} showed that $\mu(g,1)$ is in fact the optimal logarithmic Sobolev constant
on compact shrinking Ricci soliton $(M,g,f)$ for all scales and $\mu(g,\tau)$ is a continuous
function on $(0,\infty)$. Shortly after, the same conclusion for the non-compact case was
confirmed by Li and Wang \cite{[LiWa]}. In summary, we have the following sharp logarithmic
Sobolev inequality on complete gradient shrinking Ricci solitons for all scales without any
curvature assumption.
\begin{lemma}\label{logsi}
Let $(M,g, f)$ be an $n$-dimensional complete gradient shrinking Ricci soliton satisfying \eqref{Eq1}
and \eqref{Eq2}. For any compactly supported locally Lipschitz function $\varphi$ in $M$ with
\[
\int_M\varphi^2dv=1
\]
and any real number $\tau>0$,
\begin{equation}\label{LSI}
\mu+n+\frac n2\ln(4\pi)\le\tau\int_M\left(4|\nabla\varphi|^2+\mathrm{R}\varphi^2\right)dv-\int_M\varphi^2\ln \varphi^2dv-\frac n2\ln \tau,
\end{equation}
where $\mathrm{R}$ is the scalar curvature of $(M,g,f)$ and $\mu=\mu(g,1)$ is the Perelman's entropy functional.
\end{lemma}

Lemma \ref{logsi} implies a functional inequality, which is closed linked with the maximal
function of scalar curvature and the volume ratio.
\begin{theorem}\label{logeq}
Let $(M,g, f)$ be an $n$-dimensional complete gradient shrinking Ricci soliton satisfying \eqref{Eq1}
and \eqref{Eq2}. For any point $p\in M$ and for any $r>0$,
\begin{equation}\label{LSIequ}
\mu+n+\frac n2\ln(4\pi)\le 16\frac{V(p,r)}{V\left(p,\frac r2\right)}
+\frac{r^2}{V\left(p,\frac r2\right)}\int_{B(p,r)}\mathrm{R}dv
+\ln \frac{V(p,r)}{r^n},
\end{equation}
where $\mathrm{R}$ is the scalar curvature of $(M,g,f)$ and $\mu=\mu(g,1)$
is the Perelman's entropy functional.
\end{theorem}

\begin{proof}[Proof of Theorem \ref{logeq}]
Let $\psi:[0,\infty)\to[0,1]$ be a smooth cut-off function, which is supported in $[0,1]$
satisfying $\psi(t)=1$ on $[0,1/2]$ and $|\psi'|\leq 2$ on $[0,\infty)$. For any point
$p\in M$, we also let
\[
\varphi(x):=e^{-\frac c2}\psi\left(\frac{d(p,x)}{r}\right),
\]
where $c$ is some constant determined by the constraint condition $\int_M\varphi^2dv=1$. Obviously,
constant $c$ satisfies
\[
V\left(p,\frac r2\right)\le e^c\int_M\varphi^2dv=e^c
\]
and
\[
e^c=e^c\int_M\varphi^2dv=\int_M\psi^2(d(p,x)/r)dv\leq V(p,r).
\]
That is, $c$ satisfies
\[
V\left(p,\frac r2\right)\leq e^c\leq V(p,r).
\]
In the following, we will apply the above cut-off function to simplify
the sharp logarithmic Sobolev inequality in Lemma \ref{logsi}. Notice that $\varphi$
satisfies
\[
|\nabla \varphi|\leq\frac 2r\cdot e^{-\frac c2}
\]
and it is supported in $B(p,r)$.
For the first term of the right hand side of \eqref{LSI}, we estimate that
\begin{equation}\label{est1}
\begin{aligned}
4\tau\int_M |\nabla\varphi|^2dv&=4\tau\int_{B(p,r)\backslash B(p,\frac r2)} |\nabla\varphi|^2dv\\
&\le4\tau V(p,r)\frac{4}{r^2}e^{-c}\\
&\le\frac{16\tau}{r^2}\cdot\frac{V(p,r)}{V\left(p,\frac r2\right)}.
\end{aligned}
\end{equation}
For the second term of the right hand side of \eqref{LSI}, we have
\begin{equation}\label{est2}
\begin{aligned}
\tau\int_M\mathrm{R}\varphi^2 dv&\le \tau e^{-c}\int_{B(p,r)}\mathrm{R}dv\\
&\le\frac{\tau}{V\left(p,\frac r2\right)}\int_{B(p,r)}\mathrm{R}dv.
\end{aligned}
\end{equation}
Then we estimate the third term of the right hand side of \eqref{LSI}. Notice that continuous
function $H(t):=-t\ln t$ is concave with respect to $t>0$ and the Riemannian measure $dv$ is
supported in $B(p,r)$. Using the Jensen's inequality
\[
\frac{\int H(\varphi^2)dv}{\int dv}\leq H\left(\frac{\int \varphi^2 dv}{\int dv}\right)
\]
and the definition of $H$, we have that
\[
-\frac{\int_{B(p,r)}\varphi^2\ln\varphi^2dv}{\int_{B(p,r)}dv}
\leq-\frac{\int_{B(p,r)}\varphi^2dv}{\int_{B(p,r)}dv}\ln\left(\frac{\int_{B(p,r)}\varphi^2dv}{\int_{B(p,r)}dv}\right).
\]
Since $\int_{B(p,r)}\varphi^2dv=1$, the above estimate becomes
\[
-\int_{B(p,r)}\varphi^2\ln\varphi^2dv\leq\ln V(p,r).
\]
By the definition of $\varphi(x)$, we therefore get
\begin{equation}\label{est3}
-\int_M\varphi^2\ln\varphi^2dv=-\int_{B(p,r)}\varphi^2\ln\varphi^2dv
\le\ln V(p,r).
\end{equation}
Substituting \eqref{est1}, \eqref{est2} and \eqref{est3} into \eqref{LSI} gives
\[
\mu+n+\frac n2\ln(4\pi)\le\frac{16\tau}{r^2}\cdot\frac{V(p,r)}{V\left(p,\frac r2\right)}
+\frac{\tau}{V\left(p,\frac r2\right)}\int_{B(p,r)}\mathrm{R}dv
+\ln \frac{V(p,r)}{\tau^{\frac n2}}
\]
for any $\tau>0$. The conclusion follows by letting $\tau=r^2$.
\end{proof}

\section{Maximal function and volume ratio}\label{sec3}
In this section, we will apply Theorem \ref{logeq} to obtain an alternative theorem
about the lower bound for the maximal function of scalar curvature and the volume ratio
in the gradient shrinking Ricci soliton.

\vspace{.1in}

Following Topping's argument, given a Riemannian manifold $(M,g)$, for any point
$p\in M$ and $r>0$, we introduce the \emph{maximal function}
\[
M h(p,r):=\sup_{s\in(0,r]}s^{-1}\left[V(p,s)\right]^{-\frac{n-3}{2}}\left(\int_{B(p,s)} |h|dv\right)^{\frac{n-1}{2}}
\]
for any smooth function $h$ on $(M,g)$, and the \emph{volume ratio}
\[
\kappa(p,r):=\frac{V(p,r)}{r^n}.
\]

Now we give an alternative theorem. It says that the maximal function of
scalar curvature and the volume ratio in gradient shrinking Ricci solitons cannot be
simultaneously smaller than a fixed constant.
\begin{theorem}\label{maxirati}
Let $(M,g, f)$ be an $n$-dimensional complete gradient shrinking Ricci soliton satisfying \eqref{Eq1}
and \eqref{Eq2}. Then there exits a constant $\delta>0$ depending only on $n$ and $\mu$ such that for
any point $p\in M$ and for any $r>0$, at least one of the following is true:
\begin{enumerate}
 \item
$M \mathrm{R}(p,r)>\delta$;

 \item
 $\kappa(p,r)>\delta$.
\end{enumerate}
Here $\mathrm{R}(p,r)$ denotes the scalar curvature in the geodesic ball $B(p,r)$.
In particular, we can take
\[
\delta=\min\left\{w_n,\,(4\pi)^{\frac n2}e^{\mu+n-2^n\cdot17}\right\},
\]
where $\mu=\mu(g,1)$ is the Perelman's entropy functional and $w_n$ is the volume of
the unit $n$-dimensional ball in $\mathbb{R}^n$.
\end{theorem}
\begin{proof}[Proof of Theorem \ref{maxirati}]
Suppose that there exist a point $p\in(M,g, f)$ and $r>0$ such that $M \mathrm{R}(p,r)\leq\delta$
for some constant $\delta>0$. For any $0<\epsilon<1$, constant $\delta$ is defined by
\[
\delta:=\min\left\{(1-\epsilon)w_n,\,(4\pi)^{\frac n2}e^{\mu+n-2^n\cdot17}\right\},
\]
where $w_n$ is the volume of the unit $n$-dimensional ball in $\mathbb{R}^n$.
In the following we will show that $\kappa(p,r)>\delta$. If this
conclusion is not true, then we make the following

\vspace{.1in}

\textbf{Claim}. \emph{If there exist a point $p\in M$ and $r>0$ such that
$M \mathrm{R}(p,r)\leq\delta$ for some constant $\delta>0$, then for any $s\in(0,r]$,
$\kappa(p,s)\leq\delta$ implies $\kappa(p,s/2)\leq\delta$.}

\vspace{.1in}

This claim will be proved later. We now continue to prove Theorem \ref{maxirati}. We can use
the claim repeatedly and finally get that for any $m\in\mathbb{N}$,
\[
\kappa\left(p,\frac{r}{2^m}\right)\leq\delta\le (1-\epsilon)w_n,
\]
where $\epsilon$ is the sufficiently small positive constant.
But if we let $m\to\infty$, then
\[
\kappa\left(p,\frac{r}{2^m}\right)\to w_n,
\]
which contradicts the preceding inequality. So $\kappa(p,r)>\delta$ and the theorem follows.
The desired constant $\delta$ is obtained by letting $\epsilon\to 0+$.
\end{proof}

In the rest, we only need to check the above claim.

\begin{proof}[Proof of Claim]
We prove the claim by two cases according to the relative sizes of $V(p,s/2)$ and $V(p,s)$.

\emph{Case one}. Suppose that
\[
V\left(p,\frac s2\right)\leq\delta^{\frac{2}{n-1}}2^{-n}s^{\frac{2n}{n-1}}\left[V(p,s)\right]^{\frac{n-3}{n-1}}.
\]
Then,
\begin{align*}
\kappa\left(p,\frac s2\right)&:=\frac{2^n}{s^n}V\left(p,\frac s2\right)\\
&\le\delta^{\frac{2}{n-1}}s^{\frac{2n}{n-1}-n}\left[V(p,s)\right]^{\frac{n-3}{n-1}}\\
&=\delta^{\frac{2}{n-1}}(\kappa(p,s))^{\frac{n-3}{n-1}}\\
&\leq\delta^{\frac{2}{n-1}}\delta^{\frac{n-3}{n-1}}\\
&=\delta,
\end{align*}
which gives the claim.

\emph{Case Two}. Suppose that
\[
V\left(p,\frac s2\right)>\delta^{\frac{2}{n-1}}2^{-n}s^{\frac{2n}{n-1}}\left[V(p,s)\right]^{\frac{n-3}{n-1}}.
\]
Since $M \mathrm{R}(p,r)\le\delta$, by the definition of $M \mathrm{R}(p,r)$
and the scalar curvature $\mathrm{R}>0$, we get
\[
\int_{B(p,s)}\mathrm{R}dv\le\delta^{\frac{2}{n-1}}s^{\frac{2}{n-1}}\left[V(p,s)\right]^{\frac{n-3}{n-1}}
\]
for all $s\in(0,r]$. Using the assumption of Case Two, we further get
\[
\int_{B(p,s)}\mathrm{R}dv\le2^n s^{-2}V\left(p,\frac s2\right)
\]
for all $s\in(0,r]$. Substituting this into \eqref{LSIequ} and using $\kappa(p,s)\leq\delta$,
we have
\begin{align*}
\mu+n+\frac n2\ln(4\pi)&\le 16\frac{V(p,s)}{V\left(p,\frac s2\right)}
+\frac{s^2}{V\left(p,\frac s2\right)}\int_{B(p,s)}\mathrm{R}dv
+\ln \kappa(p,s)\\
&\le 16\frac{V(p,s)}{V\left(p,\frac s2\right)}+2^n+\ln\delta
\end{align*}
for all $s\in(0,r]$. By the definition of $\delta$, we notice that
\[
\ln \delta\leq\mu+n+\frac n2\ln(4\pi)-2^n\cdot17.
\]
Substituting this into the above inequality yields
\[
\frac{V(p,s)}{V\left(p,\frac s2\right)}\ge 2^n
\]
for all $s\in(0,r]$. Therefore,
\begin{align*}
\kappa\left(p,\frac s2\right)&:=\frac{2^n\cdot V\left(p,\frac s2\right)}{s^n}\\
&\leq\frac{V(p,s)}{s^n}\\
&=\kappa(p,s)\\
&\leq\delta
\end{align*}
for any $s\in(0,r]$. This completes the proof of the claim.
\end{proof}


\section{Diameter control}\label{sec4}

In this section, we will apply Theorem \ref{maxirati} to finish the proof of
Theorem \ref{Main}. The proof uses Topping's argument in \cite{[To2]},
however more delicate analysis is required to get accurate coefficient dependence
on the dimension of manifold and the Perelman's entropy functional.

\begin{proof}[Proof of Theorem \ref{Main}]
We choose $r_0>0$ sufficiently large so that the total volume of the compact
shrinking soliton is less than $\delta r_0^n$. This choice can be achieved
because the soliton is compact. Here $\delta$ is defined as in
Theorem \ref{maxirati}. Hence for any point $p\in M$, we have
\[
\kappa(p,r_0)=\frac{V(p,r_0)}{r_0^n}\leq\frac{V(M)}{r_0^n}\leq\delta,
\]
where $V(M)$ denotes the volume of $M$. By Theorem \ref{maxirati}, we conclude
that $M \mathrm{R}(p,r_0)>\delta$. By the definition of
$M \mathrm{R}(p,r_0)$, there exists $s=s(p)>0$ such that
\begin{equation}\label{defMR}
\delta<s^{-1}\big[V(p,s)\big]^{-\frac{n-3}{2}}\left(\int_{B(p,s)} \mathrm{R}dv\right)^{\frac{n-1}{2}}.
\end{equation}
By the H\"older inequality
\[
\int_{B(p,s)} \mathrm{R}dv\le\left(\int_{B(p,s)} \mathrm{R}^{\frac{n-1}{2}}dv\right)^{\frac{2}{n-1}}\cdot\left(\int_{B(p,s)}dv\right)^{\frac{n-3}{n-1}},
\]
estimate \eqref{defMR} can be reduced to
\[
\delta< s^{-1}\int_{B(p,s)} \mathrm{R}^{\frac{n-1}{2}}dv.
\]
Therefore,
\begin{equation}\label{intine}
s(p)<\delta^{-1}\int_{B(p,s(p))} \mathrm{R}^{\frac{n-1}{2}}dv.
\end{equation}

Now we pick appropriate points $p$ at which to apply the inequality \eqref{intine}.
Since $M$ is compact, we can choose $p_1, p_2\in M$ are two extremal points in $M$
such that $\mathrm{diam}(M)=dist(p_1,p_2)$. Let $\Sigma$ be a shortest geodesic
connecting $p_1$ and $p_2$. Obviously, $\Sigma$ is covered by the geodesic balls
$\{B(p,s(p))~|~p\in \Sigma\}$. By a modification of the Vitali-type covering lemma
(see Lemma 5.2 in \cite{[To2]}, or \cite{[WZ]}), there exists a countable (possibly finite)
set of points $\{p_i\in\Sigma\}$ such that the geodesic balls $\{B(p_i,s(p_i))\}$ are
disjoint, and cover at least a fraction $\rho$, where $\rho\in(0,\frac 12)$ of $\Sigma$:
\[
\rho\,\mathrm{diam}(M)\leq \sum_i 2s(p_i).
\]
Substituting \eqref{intine} into the above inequality,
\begin{equation}
\begin{aligned}\label{precisein}
\mathrm{diam}(M)&\leq\frac{2}{\rho}\sum_is(p_i)\\
&<\frac{2}{\rho}\delta^{-1}\sum_i\int_{B(p_i,s(p_i))}\mathrm{R}^{\frac{n-1}{2}}dv\\
&\leq \frac{2}{\rho}\delta^{-1}\int_M\mathrm{R}^{\frac{n-1}{2}}dv,
\end{aligned}
\end{equation}
where $\delta>0$ is a constant, depending on $n$ and $\mu$. Letting
$\rho\nearrow\frac{1}{2}$,
\[
\mathrm{diam}(M)\le4\delta^{-1}\int_M\mathrm{R}^{\frac{n-1}{2}}dv,
\]
where $\delta$ is defined in Theorem \ref{maxirati}. This proves the desired estimate.
\end{proof}

\end{document}